\documentclass[11pt]{amsart}
\usepackage{amsfonts, amstext, amsmath, amsthm, amscd, amssymb}
\usepackage{graphicx, color,  subfigure, wrapfig, overpic, url}
\usepackage{microtype}
\usepackage{import}
\usepackage[numbers,sort]{natbib}
\usepackage[hidelinks,pagebackref,pdftex]{hyperref}
\usepackage[margin=1.2in]{geometry}

\usepackage{marginnote}
\long\def\@savemarbox#1#2{\global\setbox#1\vtop{\hsize\marginparwidth 
  \@parboxrestore\tiny\raggedright #2}}
\AtBeginDocument{
   \def\MR#1{}  }

\newcommand{\ZZ}{\mathbb{Z}}

\newcommand{\from}{\colon\thinspace}

\newtheorem{theorem}{Theorem}[section]
\newtheorem{proposition}[theorem]{Proposition}
\newtheorem{lemma}[theorem]{Lemma}
\newtheorem{corollary}[theorem]{Corollary}

\theoremstyle{definition}

\newtheorem{definition}[theorem]{Definition}
\newtheorem{example}[theorem]{Example}

\newtheorem{remark}[theorem]{Remark}

\newcommand{\refthm}[1]{Theorem~\ref{Thm:#1}}
\newcommand{\reflem}[1]{Lemma~\ref{Lem:#1}}
\newcommand{\refprop}[1]{Proposition~\ref{Prop:#1}}
\newcommand{\refcor}[1]{Corollary~\ref{Cor:#1}}

\newcommand{\refdef}[1]{Definition~\ref{Def:#1}}
\newcommand{\refsec}[1]{Section~\ref{Sec:#1}}

\newcommand{\reffig}[1]{Figure~\ref{Fig:#1}}
\newcommand{\refexamp}[1]{Example~\ref{Example:#1}}

\title{State graphs and fibered state surfaces}
\author{Darlan Gir\~ao}
\author{Jessica Purcell}

\begin{document}

\maketitle

\begin{abstract}
  Associated to every state surface for a knot or link is a state graph, which embeds as a spine of the state surface. A state graph can be decomposed along cut-vertices into graphs with induced planar embeddings. Associated with each such planar graph is a checkerboard surface, and each state surface is a fiber if and only if all of its associated checkerboard surfaces are fibers. We give an algebraic condition that characterizes which checkerboard surfaces are fibers directly from their state graphs. We use this to classify fibering of checkerboard surfaces for several families of planar graphs, including those associated with 2-bridge links. This characterizes fibering for many families of state surfaces. 
\end{abstract}

\section{Introduction}

Associated to a diagram of a knot or link in the 3-sphere are several embedded spanning surfaces called \emph{state surfaces}. These arise from a choice of Kauffman state, which were introduced by Kauffman to give insight into the Jones polynomial~\cite{Kauffman}. The associated surfaces, introduced in~\cite{Ozawa} and~\cite{PPS}, arise naturally as ribbon graphs. State surfaces are known to be related to knot and link polynomials; see for example~\cite{DFKLS:JP, FKP:Guts, KalfagianniLee:Degree}. State surfaces are also related to the hyperbolic geometry of the link complement for hyperbolic links; see for example \cite{FKP:Hyperbolic, FKP:Quasifuchsian, FKP:Guts, BMPW, BurtonKalfagianni, FinlinsonPurcell}. Because of their connections with geometry, topology, and link invariants, these surfaces have recently become objects of much study in knot theory. In this paper, we consider when such surfaces can be fibers. 

If we put restrictions on the Kauffman states, there are already results that determine if a state surface is a fiber. 
For a class of states called homogeneously adequate states, work of Futer~\cite{Futer:Fiber} and Futer, Kalfagianni, and Purcell~\cite{FKP:Guts} shows that the associated state surface is a fiber if and only if a related graph, called the reduced state graph, is a tree. This answered a question of Ozawa~\cite{Ozawa}. In \cite{FKP:Guts}, it is shown that for some special states, fibering is also indicated immediately by the colored Jones polynomial. For other restrictions of state surfaces, the problem of fibering such state surfaces has been considered by Gir\~ao~\cite{Girao:Augmented} and Gir\~ao, Nogueira, and Salgueiro~\cite{GNS:FiberSurfaces}. However, we wish to find ways of classifying fibrations of \emph{all} state surfaces, without restrictions on the Kaufmann state.

The study of fibered knots and links has been ongoing since the early 1960s, introduced by work of Neuwirth~\cite{Neuwirth:Thesis} and Stallings~\cite{Stallings1962}, followed by results of Murasugi~\cite{Murasugi}. In the 1970s, further work of Stallings~\cite{Stallings} and Harer~\cite{Harer} gave insight on constructing fibered knots and links. In the 1980s, Gabai~\cite{Gabai:Fiber} gave a procedure to decide whether an oriented link is fibered, and this has had numerous applications. To date, several results exist that prove that a given knot is fibered. For example, work of Ghiggini~\cite{Ghiggini} and Ni~\cite{Ni:KFHDetectsFiber} shows that whether or not a knot is fibered can be determined using knot Floer homology. However, we are interested not in whether there exists some fibering of a knot or link, but whether a particular fixed state surface is a fiber. This is not always easily determined even if it is known that a link is fibered; see the discussion below on 2-bridge knots and links.

There do exist examples of knots and links and fixed surfaces for which fibering can be determined immediately from a diagram. These include standard diagrams of pretzel links, due to Gabai~\cite{Gabai:Fiber}, and Montesinos knots, due to Hirasawa and Murasugi~\cite{HirasawaMurasugi}. Previous results on fibering of state surfaces for restricted states are also along these lines~\cite{Futer:Fiber, FKP:Guts, Girao:Augmented, GNS:FiberSurfaces}. These are the types of results that we wish to generalize.

The first main result of this paper is \refthm{Main}, which says that whether a state surface is a fiber is completely determined by an associated graph $G_\sigma$, called the state graph, which can be quickly read from the diagram. In particular, we reduce the problem of determining fibering for a state surface with a complicated embedding into $S^3$ to the problem of determining fibering for checkerboard state surfaces. 

The next main result is \refthm{Algebraic}, which gives an algebraic condition that characterizes when a checkerboard state surface is a fiber. In particular, it suffices to prove that a homomorphism $\phi\from \pi_1(G_\sigma)\to\pi_1(S^3-G_\sigma)$ is an isomorphism, where the map $\phi$ is read directly from $G_\sigma$; see \refprop{StallingsStar}. Stallings gave a method to determine whether a homomorphism of free groups is an isomorphism, now called the Stallings folding algorithm~\cite{StallingsFold}. This has had many applications, especially in geometric group theory. By work of Touikan, the algorithm produced runs in almost linear time~\cite{Touikan}. We show here that our state graph homomorphisms feed neatly into to the Stallings folding algorithm, and thus we can read whether or not a state surface is a fiber from its state graph in almost linear time. 

We conclude the paper by giving several applications, determining many new examples of families of knots and links with fibered state surfaces. 
One additional result that comes out of this work, which may be of independent interest, is a classification of those 2-bridge link diagrams for which the (bounded) checkerboard surface is a fiber. Recall that a 2-bridge link is determined by a rational number $p/q$, and the link has several natural diagrams, determined by different continued fraction expansions of $p/q = [a_{n-1}, \dots, a_1]$ (notation described in \refsec{Applications}). Denote the diagram by $K[a_{n-1}, \dots, a_1]$. Hatcher and Thurston note that a 2-bridge knot is fibered if and only if it is isotopic to $K[\pm 2, \dots, \pm 2]$, and the fiber is isotopic to a particular spanning surface~\cite{HatcherThurston}. However, the surfaces they describe do not ever have the form of a checkerboard surface. Here, we determine exactly which values of $a_{n-1}, \dots, a_1$ give a checkerboard surface that is a fiber, i.e.\ isotopic to one of Hatcher and Thurston's fibered surfaces.

\subsection{Organization}
In \refsec{Prelim}, we review terminology used in this paper. In particular, we define Kauffman states, state surfaces, and state graphs. With terminology defined, we then state carefully our first result, \refthm{Main}. In \refsec{Murasugi}, we review important theorems on fibering due to Gabai and others, and give the proof of \refthm{Main}. In \refsec{Stallings}, the statement and proof of our algebraic characterization of fibering is given, after recalling work of Stallings. Applications are then presented in \refsec{Applications}. This section includes, for example, a complete characterization of when the bounded checkerboard surfaces of 2-bridge links are fibered, \refthm{2Bridge}.

\subsection{Acknowledgements}
We thank Jo{\~a}o Nogueira for helpful conversations, and we thank David Futer for bringing our attention to Stallings folds. 
The first author was partially supported by  Conselho Nacional de Desenvolvimento Científico e Tecnológico (CNPq) grants 446307/2014--9 and 306322/2015--3 and the CAPES Foundation. The second author was partially supported by the Australian Research Council.

\subsection{Dedication}
While this work was underway, the first author, Darlan Gir\~ao, was diagnosed with aggressive cancer, and he died before the paper was complete. Darlan was a wonderful colleague and friend, and he will be missed. This paper is dedicated to his memory.

\section{Preliminaries and main results}\label{Sec:Prelim}
Given a diagram $D(K)$ of a link $K$, at each crossing there are two choices of \emph{resolutions} for the smoothing of the crossing: an $A$-resolution or a $B$-resolution; see \reffig{Resolutions}.  By choosing a resolution at each crossing, we construct a collection of circles, called \emph{state circles}, which are the boundaries of disjoint disks, called \emph{state disks}. These state circles induce a decomposition of the plane into connected components that we call \textit{regions}. Connect the circles by edges labeled $A$ or $B$, to indicate a resolved crossing and its resolution, as in \reffig{Resolutions}.

\begin{figure}
  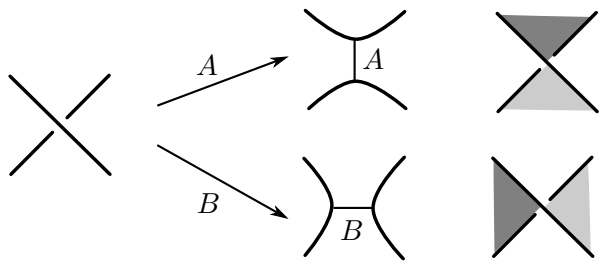
\caption{The two choices of resolutions for the split of a crossing.}
\label{Fig:Resolutions}
\end{figure}

A \emph{Kauffman state} $\sigma$ of a link diagram $D(K)$ is a choice of resolution at each crossing of $D(K)$. Given a state $\sigma$, add twisted bands connecting the state disks, one for each crossing, according to the choice of resolution. We thus obtain a surface whose boundary is the link $K$. The resulting surface $S_\sigma$ is called the \emph{state surface} of $\sigma$. For example, the Seifert surface of an oriented diagram of a link is a particular case of a state surface, where the resolution of each crossing is defined by the orientation of the link components.

The \emph{state graph} $G_{\sigma}$ has one vertex for each state disk and one edge for each band defined by the state $\sigma$. We label the edges by $A$ or $B$ according to the resolution of the respective crossings. Note that the graph $G_\sigma$ has a natural embedding in the surface $S_\sigma$ as its spine. \reffig{Fig8Example} shows an example of a diagram, a Kauffman state, and the corresponding state graph with its associated embedding.

\begin{figure}
  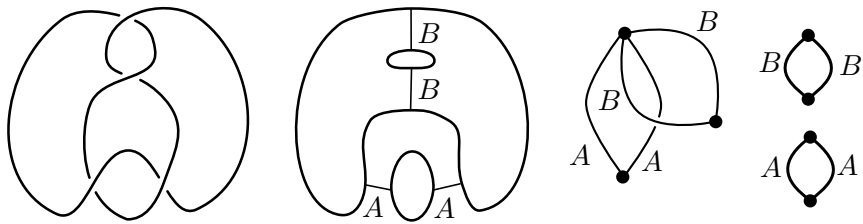
\caption{Left to right: A diagram, a Kauffman state, the corresponding state graph with its associated embedding, and the decomposition along cut-vertices into planar graphs.}
\label{Fig:Fig8Example}
\end{figure}

\begin{definition}\label{Def:CutVertex}
We say that a vertex $v$ \emph{decomposes} a graph $G$ into components $G_1, G_2$ if $G_1$ and $G_2$ contain at least one edge, and $G=G_1\cup G_2$ and $G_1\cap G_2=\{v\}$. We also say $v$ is a \emph{cut-vertex} of $G$. 
\end{definition}

\begin{lemma}\label{Lem:CutVertices}
The cut-vertices of a state graph decompose it into a finite collection of planar 2-connected graphs, with planar embeddings determined by the Kauffman state.
\end{lemma}

This decomposition is illustrated for the example of \reffig{Fig8Example} on the far right of that figure.

It is not hard to see that every state graph admits a planar embedding. However, this embedding may not be isotopic to the embedding associated with the state surface, and so it is not obviously useful from the point of view of analyzing spanning surfaces. The point of \reflem{CutVertices} is that after decomposing along cut-vertices, the remaining pieces all have a planar embedding induced from the state surface. 

\begin{proof}[Proof of \reflem{CutVertices}]
Suppose each state circle bounds a disk that is disjoint from the diagram of the link; for example this will hold if the state graph has no cut-vertices. The Kauffman state connects these disks by edges labeled $A$ or $B$ that lie in the plane, and are disjoint from the interiors of the disks. Thus when we collapse each disk to a point, the state graph we obtain remains planar.

So suppose there is a state circle such that disks on both sides meet the diagram. Then this corresponds to a cut-vertex. There must be an innermost such state circle $C$. That is, $C$ bounds a disk $D$ in the plane $S^2$ such that both $D$ and $S^2-D$ meet the diagram, but all state circles within $D$ bound disks disjoint from the diagram. Decompose the state graph along the corresponding cut-vertex. This decomposes the state graph into a graph associated with the outside of $D$ and a graph associated with the inside of $D$. On the outside of $D$, there are at most $n-1$ cut-vertices, and so the result follows by induction. On the inside of $D$, all state circles bound disks disjoint from the diagram. But now the state circle $C$ also bounds a disk $D'=S^2-D$ disjoint from the remainder of the diagram. By the previous argument, the graph we obtain by collapsing all these disks to points is planar, with embedding determined by the Kauffman state. 
\end{proof}

As in \reffig{Fig8Example}, the planar graphs determined by \reflem{CutVertices} may have multiple edges connecting the same pair of vertices, and these edges may be labeled both $A$ and $B$. We say that two edges are \emph{parallel within the plane} if the edges are ambient isotopic in the plane in the complement of all other edges and vertices of the graph. 

For multiple edges connecting the same pair of vertices that have the same label ($A$ or $B$) and that are parallel within the plane, remove all but one such edge. When all such edges are removed, the resulting graph is called the \emph{reduced state graph}. Note that multiple edges may remain in the reduced state graph, but they will either have distinct labels, or they will not be parallel. \reffig{ReducedGraph} illustrates this by example. 

\begin{figure}
  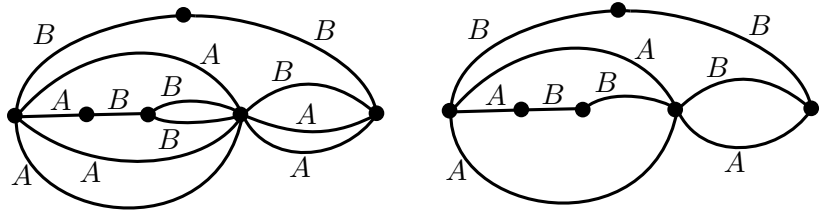
  \caption{On the left is a 2-connected state graph with a fixed planar embedding. On the right is the associated reduced state graph.}
  \label{Fig:ReducedGraph}
\end{figure}

Any planar graph with edges labeled $A$ and $B$ corresponds to a \emph{checkerboard surface} as follows. For each vertex, take a disk. For each edge, take a twisted band, with direction of twisting determined by the label $A$ or $B$; see \reffig{Checkerboard}.

\begin{figure}
  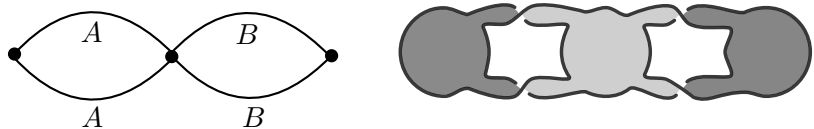
  \caption{A planar graph with edges labeled $A$ and $B$ corresponds to a checkerboard surface.}
  \label{Fig:Checkerboard}
\end{figure}

In addition to multiple edges, we may also reduce our planar state graphs by considering consecutive edges.

\begin{definition}
Two edges are called \textit{consecutive} if they share a vertex $v$ and there are no other edges incident at $v$.
\end{definition}

\begin{lemma}\label{Lem:Consecutive}
If two consecutive edges in a reduced state graph do not have the same label, then the corresponding link and state surface can be isotoped to a simpler link, removing two crossings. On the graph, this corresponds to collapsing the two edges.
\end{lemma}

\begin{proof}
\begin{figure}
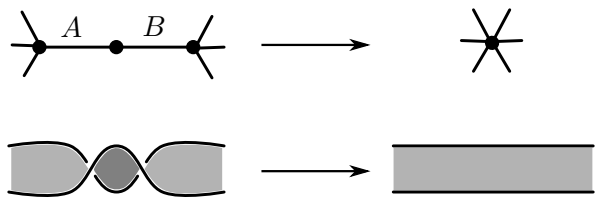
\caption{Collapsing consecutive edges with distinct labels.}
\label{Fig:Collapse}
\end{figure}
The simplification is shown in figure~\ref{Fig:Collapse}.
\end{proof}

We can now state our first main theorem.

\begin{theorem}\label{Thm:Main}
Let $K$ be a diagram of a link in $S^3$ with oriented state surface $S_\sigma$ associated to the state $\sigma$. Decompose the corresponding state graph $G_\sigma$ along cut-vertices into planar graphs, and consider the reduced planar graph, with parallel edges with the same label removed. Finally, collapse consecutive edges with distinct labels. Let $\{G_1, \dots, G_n\}$ denote the resulting collection of planar graphs. The surface $S_\sigma$ is a fiber if and only if for each $G_i$, the associated checkerboard surface is a fiber.
\end{theorem}

Note that a spanning surface for a link in $S^3$ can only be a fiber if it is orientable; thus we restrict to orientable surfaces in \refthm{Main}. Observe from the definitions that a state surface will be orientable if and only if its state graph is bipartite, so we restrict to bipartite state graphs. 

\begin{remark}
By \reflem{Consecutive}, to determine if a state surface is a fiber, it suffices to collapse consecutive edges with distinct labels. But this is done in \refthm{Main} after the state graph has been decomposed along cut-vertices and reduced. The decomposition and reduction may create new consecutive edges with distinct labels, which can then be collapsed. However, the isotopy of \reflem{Consecutive} does not extend in general across a cut-vertex, or across multiple crossings corresponding to parallel edges. 
\end{remark}

\begin{corollary}\label{Cor:SameDecomp}
Let $K_1, K_2$ be diagrams of links in $S^3$ and let $S_1,S_2$ be oriented state surfaces associated to the states $\sigma_1, \sigma_2$ of the  diagrams, respectively. Decompose the corresponding state graphs $G_1$, $G_2$ along cut-vertices into planar graphs, and for each planar graph component, take the associated reduced graph and remove consecutive edges with distinct labels. Suppose that the reduced planar graphs coming from $G_1$ can be matched in one-to-one correspondence with the reduced planar graphs coming from $G_2$, by ambient isotopy within the plane via an isotopy preserving labels $A$ and $B$. Then $S_1$ is a fiber if and only if $S_2$ is. \qed
\end{corollary}

We illustrate \refcor{SameDecomp} by an example. In figures~\ref{Fig:Fig8Example} and~\ref{Fig:HopfExample} we have two distinct links with given resolutions, each with a single cut-vertex. When we decompose, for each link we obtain two graphs that are cycles with two edges having the same label. The reduced graphs of the decomposition are then a pair of single edges, one labeled $A$ and one labeled $B$. Thus the hypotheses of \refcor{SameDecomp} are satisfied, and one is fibered if and only if the other is fibered. In this case, both examples are well-known to be fibered. Notice, however, that the links and the state surfaces are completely different (not homeomorphic).

\begin{figure}
  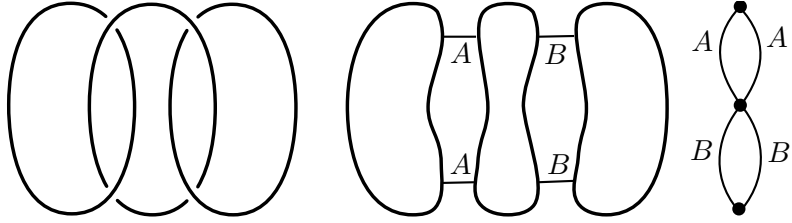
\caption{Left to right: A diagram, a Kauffman state $\sigma$ determines $H_\sigma$, and the corresponding state graph. Note the state graph is isomorphic to that in figure~\ref{Fig:Fig8Example}.}
\label{Fig:HopfExample}
\end{figure}

A Kauffman state is said to be \emph{homogeneous} if all resolutions of the diagram in each region are the same, and \emph{adequate} if there are no 1-edge loops. In \cite{Futer:Fiber} and \cite{FKP:Guts}, it was shown that a state surface associated with a homogeneous, adequate state is a fiber if and only if the reduced state graph is a tree. We obtain the following corollary.

\begin{corollary}\label{Cor:Tree}
Suppose $K$ is a link in $S^3$ with oriented state surface $S_\sigma$, and suppose the associated reduced state graph $G_\sigma$ is a tree. Then $S_\sigma$ is a fiber. 
\end{corollary}

\begin{proof}
If the reduced state graph $G_\sigma$ is a tree, it decomposes along cut-vertices into a collection of graphs made up of single edges. For each single edge, the associated checkerboard surface is the state surface of a homogeneous state, thus by \cite{Futer:Fiber, FKP:Guts}, each is fibered. Then \refthm{Main} implies the original link is also fibered. 
\end{proof}

Note that Gir{\~a}o, Nogueira, and Salgueiro also obtained a result that a particular state surface is a fiber if and only if the associated state graph is a tree~\cite{GNS:FiberSurfaces}. However, the converse to \refcor{Tree} is not true. There are examples of states for which the state graph is not a tree, and yet the state surface is still a fiber. We present examples below, such as \refexamp{FiberNotTree}. These contrast to the results of \cite{Futer:Fiber, FKP:Guts, GNS:FiberSurfaces}.

\section{Murasugi sums and cut-vertices}\label{Sec:Murasugi}

This section gives the proof of \refthm{Main}. We first recall standard definitions and results from fibered knot theory.

\begin{definition}\label{Def:MurasugiSum}
The oriented surface $T$ in $S^3$ with boundary $L$ is the \emph{Murasugi sum} of the two oriented surfaces $T_1$ and $T_2$ with boundaries $K_1$ and $K_2$ if there exists a $2$-sphere $S$ in $S^3$ bounding balls $B_1$ and $B_2$ with $T_i\subset B_i$ for $i=1,2$, such that $T=T_1\cup T_2$ and $T_1\cap T_2=D$, where $D\subset S$ is a $2k$-sided polygon (a disk). See \reffig{MurasugiSum}.  
\end{definition}

\begin{figure}
  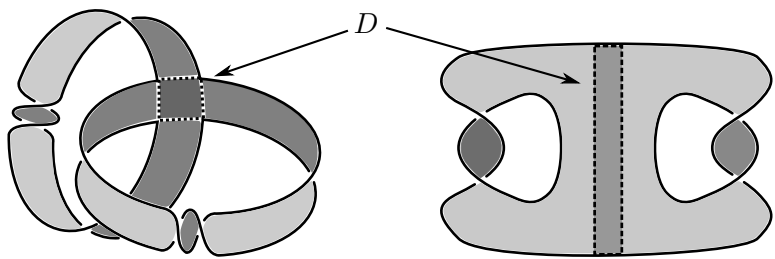
\caption{Two possible ways to obtain a surface as the Murasugi sum of two Hopf bands.} 
\label{Fig:MurasugiSum}
\end{figure}

The result concerning Murasugi sum we need is the following, due to Gabai~\cite{Gabai:Fiber}.

\begin{theorem}[Gabai]\label{Thm:Gabai}
Let $T\subset S^3$, with $\partial T=K$, be a Murasugi sum of oriented surfaces $T_i\subset S^3$, with $\partial T_i=K_i$, for $i=1,2$.  Then $S^3-L$ is fibered with fiber $T$ if and only if $S^3-K_i$ is fibered with fiber $T_i$ for $i=1,2$. \qed
\end{theorem}

\begin{lemma}\label{Lem:Murasugi}
Let $G_{\sigma}$ be a state graph embedded in $S^3$ and suppose there is a cut-vertex $v$ that decomposes $G_\sigma$ into graphs $G_1$, $G_2$. Consider the state surface $S_i$ induced by $\sigma$ and the subgraph $G_i$ of $G_\sigma$, $i=1, 2$. Then $S_\sigma$ is a fiber if and only if each of  $S_1$ and $S_2$ is a fiber.    
\end{lemma}

\begin{proof}
Each vertex $v$ of $G_\sigma$ is associated to a disk $D$ in the state surface $S_\sigma$. The fact that $v$ is a cut-vertex means that $S_1$ and $S_2$ lie in balls $B_1$, $B_2$, with $D = S_1\cap S_2$, as required by \refdef{MurasugiSum}. 
Thus the result follows from \refthm{Gabai}. 
\end{proof}

\refthm{Gabai} gives a quick proof of the following proposition.

\begin{proposition}\label{Prop:MainNotReduced}
Suppose $K$ is a link diagram with oriented state surfaces $S_\sigma$ and associated state graph $G_\sigma$ embedded as a spine in $S_\sigma$. Decompose the state graph $G_\sigma$ along cut-vertices into planar graphs. Then 
$S_\sigma$ is a fiber if and only if for each graph in the decomposition of $G_\sigma$ the associated checkerboard surface is a fiber.
\end{proposition}

\begin{proof}
The graph $G_\sigma$ decomposes along cut-vertices into planar subgraphs $\{G_{1}, \dots, G_{n}\}$ by \reflem{CutVertices}. By \reflem{Murasugi}, $S_\sigma$ is a fiber if and only if the state surfaces induced by the $G_{j}$ are fibers. 
\end{proof}

\subsection{Further decompositions of graphs}\label{Sec:Graphs}
To go from \refprop{MainNotReduced} to \refthm{Main}, we need to consider reduced graphs. The main tool is the following lemma. 

\begin{lemma}\label{Lem:MultipleEdges}
Suppose the state graph $G_\sigma$ has a planar embedding induced from the state surface $S_\sigma$. Suppose it has multiple edges connecting a pair of vertices, and suppose that those edges are parallel (ambient isotopic) in the plane. 
\begin{enumerate}
\item If multiple parallel edges have the same label, then $S_\sigma$ is a fiber if and only if the same is true for the state surface corresponding to the graph with all but one of the edges removed.
\item If two of the parallel edges have distinct labels, then $S_\sigma$ is not a fiber.
\end{enumerate}
\end{lemma}

\begin{proof}
In both cases, we may use a Murasugi sum to decompose the surface. In the first case, parallel edges with the same label correspond to a Murasugi sum of a Hopf band; see \reffig{HopfDecompose}. The twisted annulus bounded by a Hopf band is well-known to be fibered. 
\begin{figure}
\includegraphics{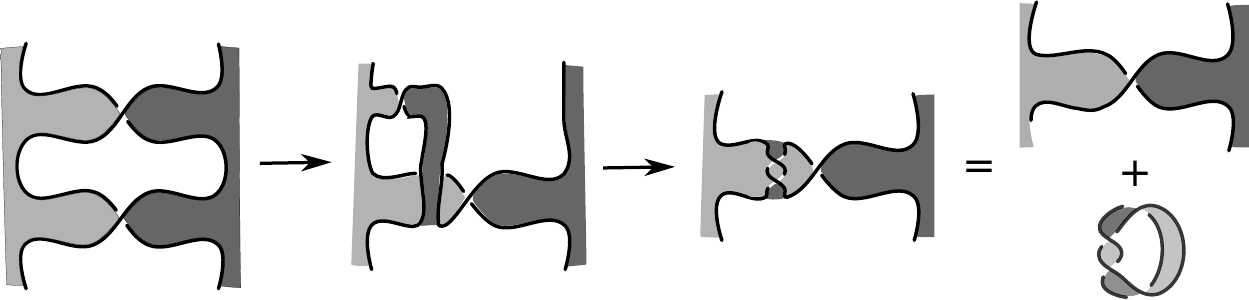}
\caption{A pair of eges with the same label correspond to the decomposition of a Hopf band.}
\label{Fig:HopfDecompose}
\end{figure}

In the second case, parallel edges with different labels correspond to a Murasugi sum of an annulus; see \reffig{AnnulusDecompose}. The untwisted annulus in $S^3$ is well-known not to be fibered. 
Then the result follows by Gabai's theorem, \refthm{Gabai}. \qedhere
\begin{figure}
\includegraphics{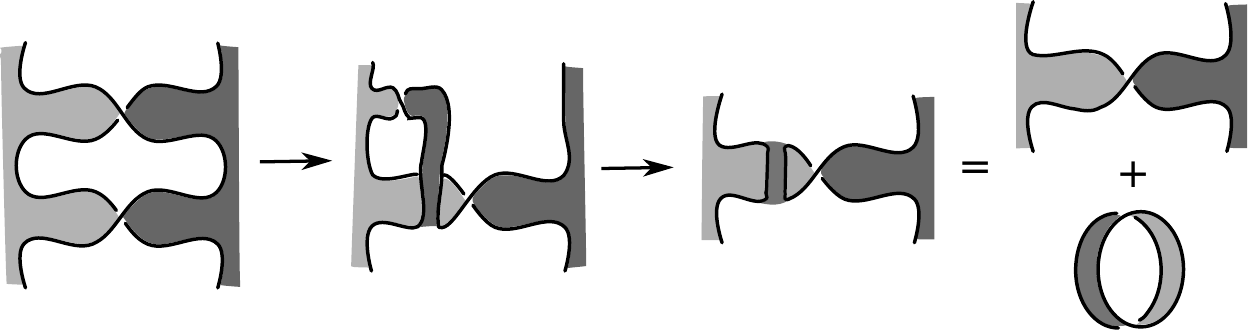}
\caption{A pair of eges with distinct label correspond to the decomposition of an annulus.}
\label{Fig:AnnulusDecompose}
\end{figure}
\end{proof}

\begin{corollary}\label{Cor:MultiEdgesDifferentLabel}
Suppose $K$ is a link in $S^3$ with oriented state surface $S_\sigma$, and suppose the associated state graph $G_\sigma$ has parallel edges with distinct labels. Then $S_\sigma$ is not a fiber. 
\end{corollary}

\begin{proof}
Parallel edges with distinct labels will remain parallel edges with distinct labels after decomposing along cut-vertices. Then the result follows from \refthm{Main} and \reflem{MultipleEdges}. 
\end{proof}

\begin{proof}[Proof of \refthm{Main}]
By \refprop{MainNotReduced}, $S_\sigma$ is a fiber if and only if for each graph in the decomposition of its state graph along cut-vertices, the associated checkerboard surface is a fiber. By \reflem{MultipleEdges}, that checkerboard surface is a fiber if and only if the checkerboard surface associated with the corresponding reduced graph is a fiber. By \reflem{Consecutive}, the checkerboard surface associated with the reduced graph is a fiber if and only if the same is true of the checkerboard surface associated with the graph obtained by collapsing consecutive edges that do not have the same label.
\end{proof}

\section{Stallings criterion}\label{Sec:Stallings}

By \refthm{Main}, we may reduce the task of determining which state surfaces are fibers to considering checkerboard surfaces. To determine which checkerboard surfaces are fibers, we will apply the following theorem of Stallings~\cite{Stallings}. 

\begin{theorem}[Stallings]\label{Thm:Stallings}
Let $S\subset S^3$ be a compact, connected, oriented surface with nonempty boundary $\partial S$. Let $S\times[-1,1]$ be a regular neighborhood of $S$ and let $S^+=S\times\{1\}\subset S^3-S$. Let $f=\varphi|_{S}$, where $\varphi\from S\times[-1,1]\to S^+$ is the projection map. Then $S$ is a fiber for the link $\partial S$ if and only if the induced map $f_* \from \pi_1(S) \to \pi_1(S^3-S)$ is an isomorphism.   
\end{theorem}

As a consequence of \refthm{Stallings}, we obtain a tool to determine fibering from state graphs:

\begin{theorem}\label{Thm:Algebraic}
Let $S_\sigma$ be a state surface with state graph $G_\sigma$ embedded as a spine, and let $x$ be a basepoint on $G_\sigma\subset S_\sigma$. Let $f = \varphi|_S$, where $\varphi\from S_\sigma\times[-1,1]\to S_\sigma^+$ is the projection map, and let $y=f(x)$. Then the map $f$ uniquely determines a map $\phi\from \pi_1(G_\sigma,x)\to \pi_1(S^3-G_\sigma, y)$, and $S_\sigma$ is a fiber if and only if the map $\phi$ is an isomorphism of groups.
\end{theorem}

\begin{proof}
The graph $G_{\sigma}$ lies naturally in $S_{\sigma}$ as a spine, and therefore the inclusion $G_{\sigma}\to S_{\sigma}$ induces an isomorphism $\pi_1(G_{\sigma})\to \pi_1(S_{\sigma})$. Similarly, there is an isomorphism
\[ \pi_1(S^3-N(S_{\sigma}))\to \pi_1(S^3-S_{\sigma})\]
induced by the deformation retraction $S^3-S_{\sigma}$ to $S^3-N(S_{\sigma})$, where $N(\cdot)$ denotes a regular neighborhood. Deformation retractions (and inclusions) induce the following additional isometries:
\[
\pi_1(S^3-N(S_{\sigma}))\cong \pi_1(S^3-N(G_{\sigma}))\cong \pi_1(S^3-G_{\sigma}).
\]
Finally note that we may also assume
\[ f(S_{\sigma})=S_{\sigma}^+\subset S^3-N(S_{\sigma}). \]

Consider the map
\[ \phi\from \pi_1(G_{\sigma}, x)\to\pi_1(S^3-G_{\sigma}, y) \]
induced by the following compositions:
\[
G_{\sigma}\to S_{\sigma} \to (S^3-S_{\sigma}) \to (S^3-N(S_{\sigma})) \to (S^3-N(G_{\sigma})) \to (S^3-G_{\sigma}),\]
where the first map is the inclusion, the second map is the map $f$, the third map is given by deformation retraction, and the fourth and fifth maps are given by inclusion. Thus the map $\phi$ is an isomorphism if and only if the map $f_*$ in Stallings theorem is an isomorphism, and $S_\sigma$ is a fiber if and only if $\phi$ is an isomorphism. 
\end{proof}

We call the map $\phi$ in \refthm{Algebraic} the \emph{Stallings map}. We need to determine the effect of the Stallings map. We start by considering the form of the complement $S^3-S_\sigma$. 

Suppose that a state graph $G_\sigma$, embedded as the spine of a state surface $S_\sigma$, is planar with no cut-vertices. Then $S^3-S_\sigma$ is a handlebody built as follows.
\begin{itemize}
\item There is one 0-handle $H_+$ above the plane of projection $P$, and one $H_-$ below.
\item For each region $R_j$ of $P-G_\sigma$, there is a 1-handle $T_j$ running through that region with one end on $H_+$ and the other on $H_-$. 
\end{itemize}

\begin{lemma}\label{Lem:StallingsLocal}
Suppose that the state graph $G_\sigma$, embedded as the spine of a state surface $S_\sigma$, is planar and has no cut-vertices. Suppose that $S_\sigma$ is oriented, so $G_\sigma$ is bipartite, with vertices labeled $+$ and $-$, and edges labeled $A$ and $B$ according to the choice of resolution. Let $\gamma$ be a directed arc in $G_\sigma\subset S_\sigma$ that runs over an edge.
Denote the 0-handles of $S^3-S_\sigma$ by $H_+$ and $H_-$, as above. Let $T_\ell$ denote the 1-handle running through the region to the left of the edge (determined by the direction of $\gamma$), and let $T_r$ denote the 1-handle running through the region to the right. 
Then the effect of the map $f=\varphi|_S$, where $\varphi\from S\times[-1,1]\to S^+$ is the projection map, is as follows.
  \begin{enumerate}
  \item If $\gamma$ is a monotonic arc running from a vertex labeled $+$ to one labeled $-$ over an edge labeled $A$, then $f(\gamma)$ runs from $H_+$ to $H_-$ along the 1-handle $T_\ell$.
  \item If $\gamma$ runs from a vertex labeled $+$ to one labeled $-$ over an edge labeled $B$, then $f(\gamma)$ runs from $H_+$ to $H_-$ along the 1-handle $T_r$.
  \item If $\gamma$ runs from a vertex labeled $-$ to one labeled $+$ over an edge labeled $A$, then $f(\gamma)$ runs from $H_-$ to $H_+$ along the 1-handle $T_r$.
  \item If $\gamma$ runs from a vertex labeled $-$ to one labeled $+$ over an edge labeled $B$, then $f(\gamma)$ runs from $H_-$ to $H_+$ along the 1-handle $T_\ell$.
  \end{enumerate}
\end{lemma}

\begin{proof}
If $\gamma$ starts on a vertex labeled $+$ and runs to one labeled $-$, then $f(\gamma)$ begins at $H_+$ and ends at $H_-$. Similarly if $\gamma$ starts at a vertex labeled $-$ and runs to one labeled $+$, then $f(\gamma)$ begins at $H_-$ and ends at $H_+$. The 1-handles that $f(\gamma)$ meets are determined by the twisting of $S_\sigma$ along the twisted band attached at the crossing of the edge, with twisting in one direction for an edge labeled $A$ and in the other direction for an edge labeled $B$. The effects are shown in \reffig{MapF}.
\end{proof}

\begin{figure}
  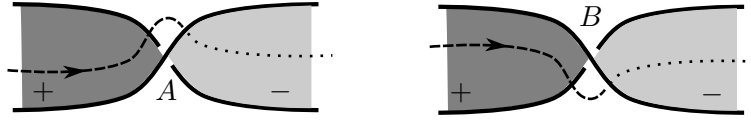
  \caption{Effect of the map $f$ on the edges  of $G_{\sigma}$ for $\gamma$ running from $+$ to $-$. For $-$ to $+$, simply change the direction $\gamma$.}
  \label{Fig:MapF}
\end{figure}

Suppose $G_\sigma$ is a state graph that comes out of \refthm{Main}, namely it is planar with no cut-vertices. To determine whether the associated state surface $S_\sigma$ is a fiber, we choose generators for $\pi_1(G_\sigma,x)$ and consider their image under the Stallings map $\phi$.

Because $S_\sigma$ is orientable, $G_\sigma$ is bipartite, i.e.\ every closed loop has even length. Label each vertex either by $+$ or $-$ in an alternating fashion. Label the basepoint $x$ in $G_\sigma$ with $+$ to obtain a well-defined labeling of vertices of $G_\sigma$.

The graph $G_{\sigma}$ divides the plane into regions. The unbounded region is denoted by $R_0$. The  bounded ones are  denoted by $R_1, \dots, R_n$.

Corresponding to each region $R_j$, there is a 1-handle $T_j$ with its ends on the 0-handles $H_+$ and $H_-$, above and below the plane of projection, respectively. 
Since the basepoint $x$ of $G_\sigma$ is labeled $+$, its image $y=f(x)$ lies on $H_+$. For each bounded region $R_j$, $j=1, \dots,n$, define a curve $u_j$ based at $y$ as follows. The curve $u_j$ leaves the basepoint $y$ and runs monotonically down the 1-handle $T_j$ through the region $R_j$ to the 0-handle $H_-$. It then runs up the 1-handle $T_0$ through the unbounded region to connect to $H_+$, and back to the basepoint $y$. Then $\{u_1, \dots, u_n\}$ form a generating set for $\pi_1(S^3-G_\sigma, x)$.

Given a loop $u \subset S^3-G_{\sigma}$ based at $y$, its homotopy class in $\pi_1(S^3-G_{\sigma}, y)$ is given by a word in the letters $u_1,...,u_n$ as follows: Starting at $y$, move along $u$ acording to the choice of orientation. If $u$ crosses a region $R_i$ from above to below, then write the letter $u_i$. 
If $u$ crosses a region $R_i$ from below to above, then write the letter $u_i^{-1}$. Going around $u$ once gives a word which represents its homotopy class in $\pi_1(S^3-G_{\sigma},y)$. If the  region $R_0$ is crossed, then write no letters.

\begin{definition}\label{Def:Labeling}
The following procedure defines a new labeling on the graph $G_\sigma$, with labels the letters $u_i$. Direct each edge from $+$ to $-$. If the edge is labeled $A$, assign to the directed edge the letter $u_i$, where $R_i$ is the region to the left of the directed edge, or $1$ if the region to the left is unbounded. If the edge is labeled $B$, assign the letter $u_j$, where $R_j$ is the region to the right of the directed edge, or $1$ if the region to the right is unbounded. 
\end{definition}

\begin{proposition}\label{Prop:StallingsStar}
  Suppose that the state graph $G_\sigma$, embedded as the spine of an oriented state surface $S_\sigma$, is planar and has no cut-vertices. The effect of the Stallings map $\phi \from \pi_1(G_\sigma, x) \to \pi_1(S^3-G_\sigma, y)$ is as follows. A generator $\gamma_i$ of $\pi_1(G_\sigma,x)$ is represented by a word in the edges of $G_\sigma$. The image of $\gamma_i$ under $\phi$ is represented by a word in the letters $\{u_1^{\pm 1}, \dots, u_n^{\pm 1}\}$, where each edge $e$ contributes a letter as follows.

  Consider the labeling of \refdef{Labeling}. 
  \begin{enumerate}
  \item If $u_j$ is the label on $e$, and $\gamma_i$ runs along $e$ from $+$ to $-$, write $u_j$.
  \item If $u_j$ is the label on $e$, and $\gamma_i$ runs along $e$ from $-$ to $+$, write $u_j^{-1}$. 
  \item If $e$ is labeled $1$ and $\gamma_i$ runs along $e$, write no letters.    
  \end{enumerate}
\end{proposition}

\begin{proof}
This follows from \reflem{StallingsLocal} and our choice of generators of $\pi_1(G_\sigma, x)$ and $\pi_1(S^3-G_\sigma, y)$ defined above. 
\end{proof}

\begin{corollary}\label{Cor:AlgebraicFibering}
Suppose that the state graph $G_\sigma$, embedded in the spine of an oriented state surface $S_\sigma$, is planar with no cut-vertices, and let $\gamma_1, \dots, \gamma_n$ denote the generators of $\pi_1(G_\sigma, x)$ as above.  Then $S_\sigma$ is a fiber if and only if the group generated by $\phi(\gamma_1),\dots, \phi(\gamma_n)$ is isomorphic to $\pi_1(S^3-G_\sigma, y)$. \qed
\end{corollary}

\begin{example}\label{Example:FiberNotTree}
Consider the link in \reffig{NewExample}.
\begin{figure}[h]
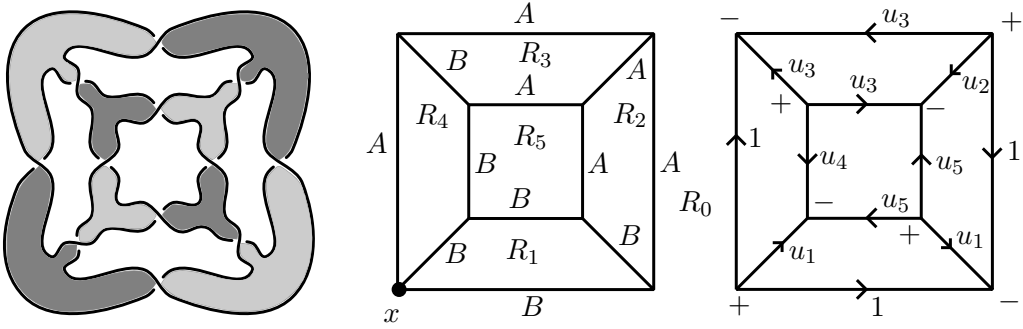
\caption{A checkerboard state surface that is a fiber, and its associated state graph.}
\label{Fig:NewExample}
\end{figure}
The base point $x$ is indicated. The labeling of \refdef{Labeling} is shown on the far right. Let $\{\gamma_1, \dots, \gamma_5\}$ denote generators of $\pi_1(G_\sigma,x)$ as follows. The curve $\gamma_i$ begins at $x$, traverse a path $\beta_i$ to the boundary of the region $R_i$, runs counter clockwise around $R_i$, then traverses the path $\beta_i^{-1}$ back to $x$. More specifically, we choose $\beta_2$ to be the arc given by the lower horizontal edge starting at $x$, we choose $\beta_3$ to be the arc given by the lower horizontal edge followed by the vertical right-most edge, and we choose $\beta_5$ to be the diagonal edge with endpoint on $x$. The arcs $\beta_1$ and $\beta_4$ are trivial. Following the recipe of \refprop{StallingsStar}, we obtain:
\[
\begin{cases}
\phi(\gamma_1)=u_1^{-1}u_5u_1^{-1}\\
\phi(\gamma_2)=u_2u_5^{-1}u_1\\
\phi(\gamma_3)= u_3u_2^{-1}\\
\phi(\gamma_4)= u_1u_4^{-1}u_3\\
\phi(\gamma_5)=u_1u_3^{-1}u_4u_1^{-1}
\end{cases}
\] 
It is possible to see directly that this map is an isomorphism; its inverse given by:
\[
\begin{cases}
\phi^{-1}(u_1)=\gamma_5\gamma_4\\
\phi^{-1}(u_2)=\gamma_2\gamma_1\gamma_5\gamma_4\\
\phi^{-1}(u_3)=\gamma_3\gamma_2\gamma_1\gamma_5\gamma_4\\
\phi^{-1}(u_4)=\gamma_3\gamma_2\gamma_1\gamma_5^2\gamma_4\\
\phi^{-1}(u_5)=\gamma_5\gamma_4\gamma_1\gamma_5\gamma_4
\end{cases}
\]
Therefore the state surface shown is a fiber for the corresponding link. 
\end{example}

\subsection{Stallings folding}

In \cite{StallingsFold}, Stallings describes a procedure that can determine whether a homomorphism of free groups is an isomorphism. See also Kapovich and Myasnikov~\cite{KapovichMyasnikov}. 

\begin{definition}
  An \emph{edge folding} of a directed labeled graph $\Gamma$ is a new directed labled graph obtained by identifying two edges $e_1$ and $e_2$ of $\Gamma$ that share the same vertex $v$, and have the same label and the same direction incident to $v$. A directed labeled graph is \emph{folded} if at each vertex $v$ there is at most one edge with a given label and incidence starting (or terminating) at $v$. 
\end{definition}

A folded graph admits no edge foldings.

Fix a basepoint on a directed labeled graph. The fundamental group of the graph is determined by concatination of paths; group elements are determined by writing words in the labels traversed by a path, where if an edge $e$ has label $a$ and is traversed in the direction of $e$, we write $a$, and if in the opposite direction, we write $a^{-1}$. Two words are \emph{freely reduced} if two labels $a$ and $a^{-1}$ never appear consecutively. Note this is a way of representing the fundamental group as a subgroup of a free group; the labeled graph is a covering space corresponding to the subgroup. We call this the \emph{group defined on the graph}. 

The following is straightforward; see \cite{KapovichMyasnikov}. 

\begin{proposition}\label{Prop:StallingsFold}
  Let $\Gamma_1$ and $\Gamma_2$ be two directed labeled graphs. If $\Gamma_2$ is obtained from $\Gamma_1$ by an edge folding, then the two groups defined on the graphs are isomorphic. If $\Gamma_2$ is folded, then for any loop $\ell$ in $\Gamma_2$, the corresponding word is freely reduced. 
\end{proposition}

Starting with any graph, perform a sequence of edge foldings, strictly reducing the number of edges, until the process terminates in a folded graph. Touikan showed that there is an algorithm to fold a graph that runs in almost linear time~\cite{Touikan}.

\begin{corollary}\label{Cor:StallingsFold}
  Suppose $G_\sigma$ is a planar state graph with no cut vertices that is embedded as the spine of an oriented state surface $S_\sigma$. Suppose $\pi_1(G_\sigma)$ is the free group on $n$ letters. Let $\Gamma$ be a directed labeled graph obtained from $G_\sigma$ by the process of \refdef{Labeling}. Let $\Gamma'$ denote the result of $\Gamma$ after folding. Then the state surface $S_\sigma$ is a fiber if and only if $\Gamma'$ is a rose with $n$ petals. 
\end{corollary}

\begin{proof}
The graph $\Gamma$ encodes the image of Stallings' map $\phi\from\pi_1(G_\sigma)\to\pi_1(S^3-G_\sigma)$. 
Stallings folding implies that the graph $\Gamma'$ is a rose with $n$ petals if and only if $\phi$ is surjective.
Because free groups are Hopfian, this holds if and only if $\phi$ is an isomorphism. The result follows from \refthm{Algebraic}.
\end{proof}

\reffig{Folding} shows the folding process applied to the directed labeled graph of \reffig{NewExample}. Note the fact that the state surface of that example is a fiber now follows from \refcor{StallingsFold}. 

\begin{figure}
  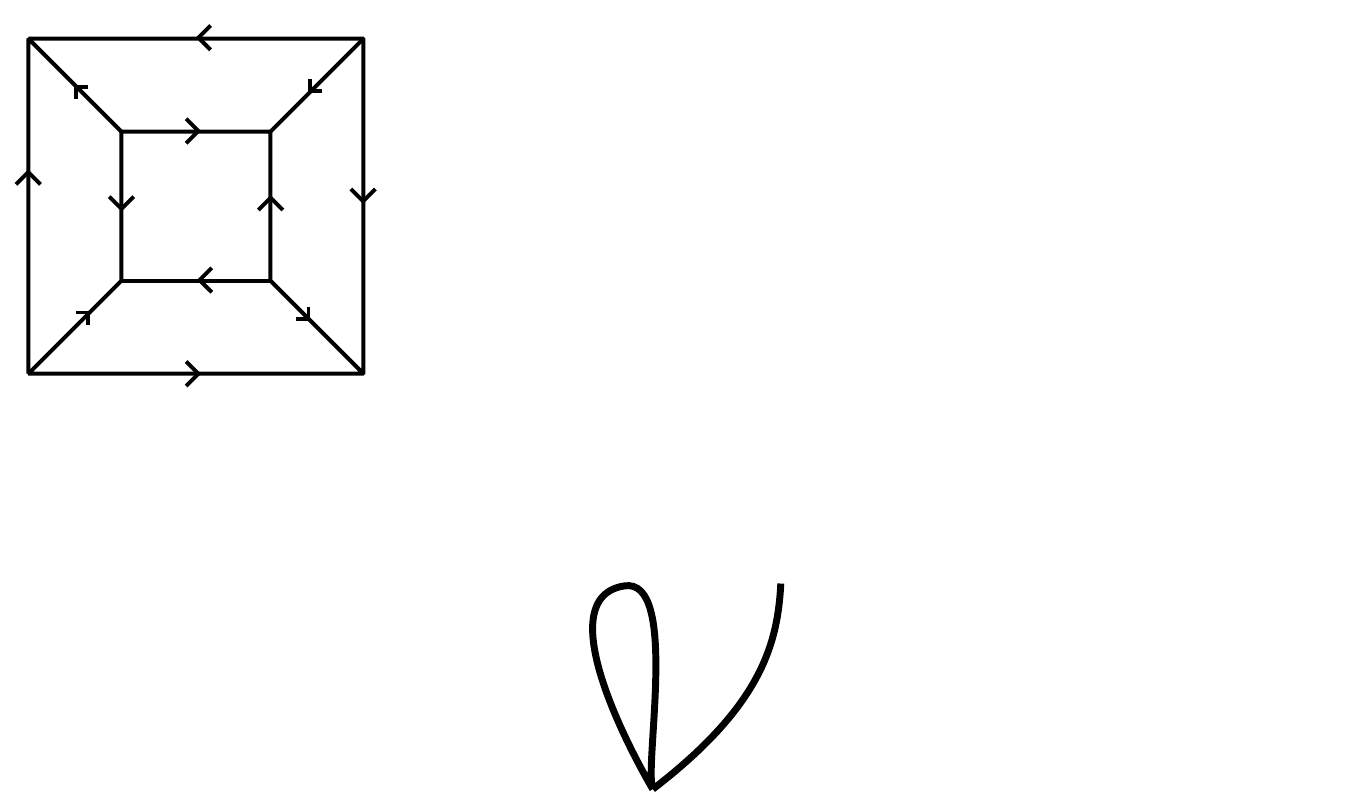
  \caption{The process of folding applied to \refexamp{FiberNotTree}.}
  \label{Fig:Folding}
\end{figure}


\section{Applications}\label{Sec:Applications}

In this section, we consider some simple graphs and completely classify when they correspond to state surfaces that are fibers.

\subsection{The state graph or reduced state graph is a tree}
This is the case of \refcor{Tree}. We have seen that the state surface is always a fiber. 

\subsection{The reduced state graph decomposes into cycles}

\begin{lemma}\label{Lem:Cycle}
Suppose $K$ is a link with state surface $S_\sigma$. Suppose that the reduced state graph corresponding to $S_\sigma$ is a cycle. Then $S_\sigma$ is a fiber if and only if there are exactly two more edges in the cycle labeled $A$ than $B$, or vice versa.
\end{lemma}

\begin{proof}
  Let $G$ denote the reduced state graph; the fact that $G$ is a cycle means that the (checkerboard) state surface $S$ induced by $G$ is an annulus. The original state surface $S_\sigma$, whose state graph may have multiple parallel edges with the same label, is obtained by taking Murasugi sums of Hopf bands with the annulus $S$, as in \reflem{MultipleEdges}. The graph $G$ is a string of consecutive edges.  Untwist $S$ as in \reflem{Consecutive}, obtaining a new graph $G'$ for which all edges have the same label. Now consider the Stallings map $\phi$. Because the surface is orientable, $G'$ has an even number of edges. We have $\pi_1(G) \cong \pi_1(G') \cong \ZZ$ and $\pi_1(S^3-G)\cong\pi_1(S^3-G') \cong \ZZ$. By \refprop{StallingsStar}, the image of the generator under $\phi$ is given by $\phi(\gamma)=u^{\pm r}$, where $2r$ is the number of edges in $G'$, and the sign is positive or negative depending on whether the labels of $G'$ are $A$ or $B$. Note $\phi$ is an isomorphism if and only if $r=1$. Thus the original state surface is a fiber in this case if and only if $G'$ has exactly two edges, meaning there are exactly two more edges of the reduced graph $G$ that are labeled $A$ than those labeled $B$, or vice versa.
\end{proof}

Note that the unbounded checkerboard surfaces for pretzel links fall into the class of the surfaces satisfying the hypotheses of \reflem{Cycle}. Gabai treats these as type~III surfaces in \cite{Gabai:Fiber}. Gabai has more conditions on when these links fiber. However, note that in the cases distinct from \reflem{Cycle}, the fibered surface is not the given state surface.

\subsection{Graphs the shape of a theta and pretzel links}

A next simple class of graphs to consider are those that have the shape of a $\Theta$: that is, there are three collections of consecutive edges meeting two vertices; we call each of the three collections of consecutive edges a \emph{strand}. There will be $p_1$ vertices on the first strand, $p_2$ on the second, and $p_3$ on the third, where either all $p_i$ are even or all are odd to ensure the graph is bipartite. We may also assume, after reducing, that all edges on each strand have the same label. The induced checkerboard state surface will be the bounded checkerboard surface of a pretzel knot with three strands, and $p_i$ crossings on each strand. 
In fact we may generalize to $n$ strands: each strand has $p_i$ vertices, all edges on a strand have the same label, and all the $p_i$ are even or all are odd. Call such a graph a \emph{generalized theta graph}. The induced checkerboard surfaces are bounded checkerboard surfaces for pretzel links. 

The fibering of such checkerboard surfaces has been completely classified. When $n=3$, this follows from work of Crowell and Trotter~\cite{CrowellTrotter}. For other values of $n$, partial results were obtained by Parris~\cite{Parris}, and independently by Goodman and Tavares~\cite{GoodmanTavares}, and Kanenobu~\cite{Kanenobu}. Their results are summarized and generalized by Gabai~\cite{Gabai:Fiber}.
The full result is the following.

\begin{theorem}[Bounded checkerboard surfaces of pretzel links]
Let $p_1, \dots, p_n$ be nonzero integers. The pretzel link determined by $(p_1, \dots, p_n)$ crossings has a bounded checkerboard surface that is a fiber if and only if one of the following holds:
\begin{enumerate}
\item each $p_i=\pm 1$ or $\mp 3$ and some $p_i=\pm 1$.
\item $(p_1, \dots, p_n) = \pm (2, -2, 2, -2, \dots, 2, -2, \ell)$ for $\ell\in\ZZ$ (here $n$ is odd).
\item $(p_1, \dots, p_n) = \pm (2, -2, 2, -2, \dots, -2, 2, -4)$ (here $n$ is even).
\end{enumerate}
\end{theorem}

\begin{corollary}\label{Cor:ThetaGraph}
Suppose a reduced state graph has the shape of a generalized theta graph. Reduce further to remove consecutive edges without the same label. Then the associated checkerboard surface is a fiber if and only if one of the following holds:
\begin{enumerate}
\item Each strand has $0$ or $2$ vertices, and some strand has $0$ vertices. All strands with $0$ vertices are labeled $A$ or are all labeled $B$; strands with $2$ vertices are exactly the opposite (all $B$ or all $A$). 
\item There are an even number of strands, each strand except the last has $1$ vertex, all strands except the last alternate being labeled $A$ and $B$, and the final strand has any number of vertices with any label $A$ or $B$ on all edges of the strand.
\item There are an odd number of strands, each strand except the last has $1$ vertex, all strands alternate being labeled $A$ and $B$, and the final strand has three vertices.
\end{enumerate}
\end{corollary}

\subsection{Checkerboard surfaces of two-bridge links}

We conclude with one final family of reduced state graphs whose fibering we can completely classify: those that arise from diagrams of 2-bridge knots. Essential surfaces in 2-bridge knots were classified by Hatcher and Thurston~\cite{HatcherThurston}, and they note in that paper that those spanning surfaces that are fibers are exactly those that embed in a diagram $K[a_{n-1}, \dots, a_1]$ with each $a_j=\pm 2$ (notation described below). However, the surfaces in \cite{HatcherThurston} are not isotoped to be checkerboard surfaces. While essential checkerboard surfaces must appear in their theorem, a nontrivial isotopy of the entire link is required move a standard checkerboard surface into one of the forms of their results. Since state graphs induce diagrams and checkerboard surfaces, we are interested in classifying those that lead to fibered surfaces without having to isotope into a form required by~\cite{HatcherThurston}.

In this section, we classify all state graphs corresponding to 2-bridge knots and links that induce checkerboard surfaces that are fibers. Note that while the surfaces must be isotopic to one in $K[\pm 2, \dots, \pm 2]$ as in~\cite{HatcherThurston}, the diagrams we obtain here are much broader.

We begin by recalling results and notation. Recall that \emph{2-bridge knot} or link is obtained by taking the closure of a rational tangle, determined by a rational number $p/q$; see for example \cite{Murasugi:Book}. For any continued fraction
\[ \frac{p}{q} = [a_n, a_{n-1}, \hdots, a_1] =  a_n + \cfrac{1}{a_{n-1} + \cfrac{1}{\ddots \, + \cfrac{1}{a_1}}} \]
we obtain a diagram of the knot with $(n-1)$ twist regions, with $|a_i|$ crossings in the $i$-th twist region, and with signs of the crossings equal to the sign of $a_i$ if $i$ is even, and $-a_i$ if $i$ is odd. Denote this diagram by $K[a_{n-1}, \dots, a_1]$. See \reffig{2BridgeDiagram}. Note that the value of $a_n$ does not affect the knot, and it can be removed from the diagram, so we omit it from the notation. Note also that we may assume $|a_i|\geq 0$, since diagrams with $0$ crossings in a twist region may be described instead by fewer twist regions, with crossings above and below merged. Similarly we may assume $|a_1|\geq 2$ and $|a_{n-1}|\geq 2$. 

\begin{figure}
  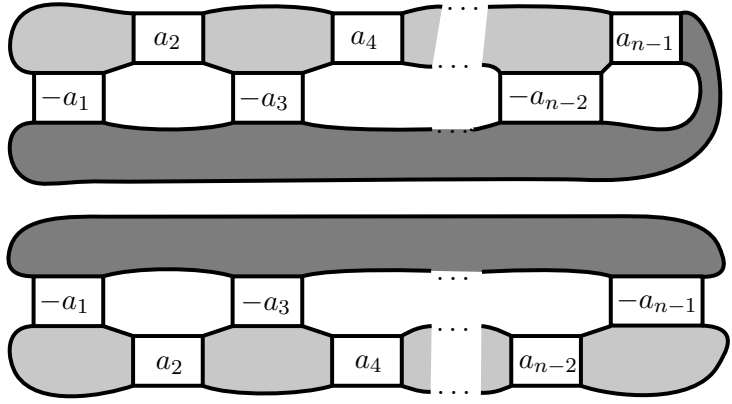
  \caption{The diagram of $K[a_{n-1}, \dots, a_1]$. Top: $n$ odd; bottom: $n$ even. Box labeled $\pm a_i$ denotes a (horizontal) twist region with $|a_i|$ crossings, with sign of the crossings equal to that of $\pm a_i$.}
  \label{Fig:2BridgeDiagram}
\end{figure}

Associated with each such diagram are two checkerboard surfaces. We will be interested in the bounded checkerboard surface, as shown in \reffig{2BridgeDiagram}.

Now consider the state graph of this checkerboard surface. There will be one vertex corresponding to the large disk in the diagram, which we will label $+$ and denote as the basepoint. For each odd $a_{2j+1}$, the basepoint will be adjacent to $b_{2j+1}=|a_{2j+1}|$ parallel multi-edges, with each edge labeled $A$ if $a_{2j+1}>0$ and $B$ otherwise. The opposite endpoint of the $b_{2j+1}$ edges will be connected to the endpoint of the $b_{2j-1}$ edges via a strand of $b_{2j} = |a_{2j}|$ consecutive edges, all labeled $A$ if $a_{2j}>0$ and $B$ otherwise.
If $n=2k+3$ is odd, the basepoint is also connected to the other endpoint of the $b_{2k+1}$ multi-edges by a sequence of $|a_{2k+2}|:= b_{2k+2}$ consecutive edges, labeled $A$ if $a_{2k+1}>0$ and $B$ otherwise. See \reffig{2BridgeStateGraph}, left. The reduced state graph replaces each group of $b_{2j-1}$ parallel multi-edges with a single edge, denoted $e_j$. See \reffig{2BridgeStateGraph}, right. By \reflem{MultipleEdges}, the checkerboard surface is a fiber if and only if the induced checkerboard surface of the reduced graph is a fiber. 

\begin{figure}
  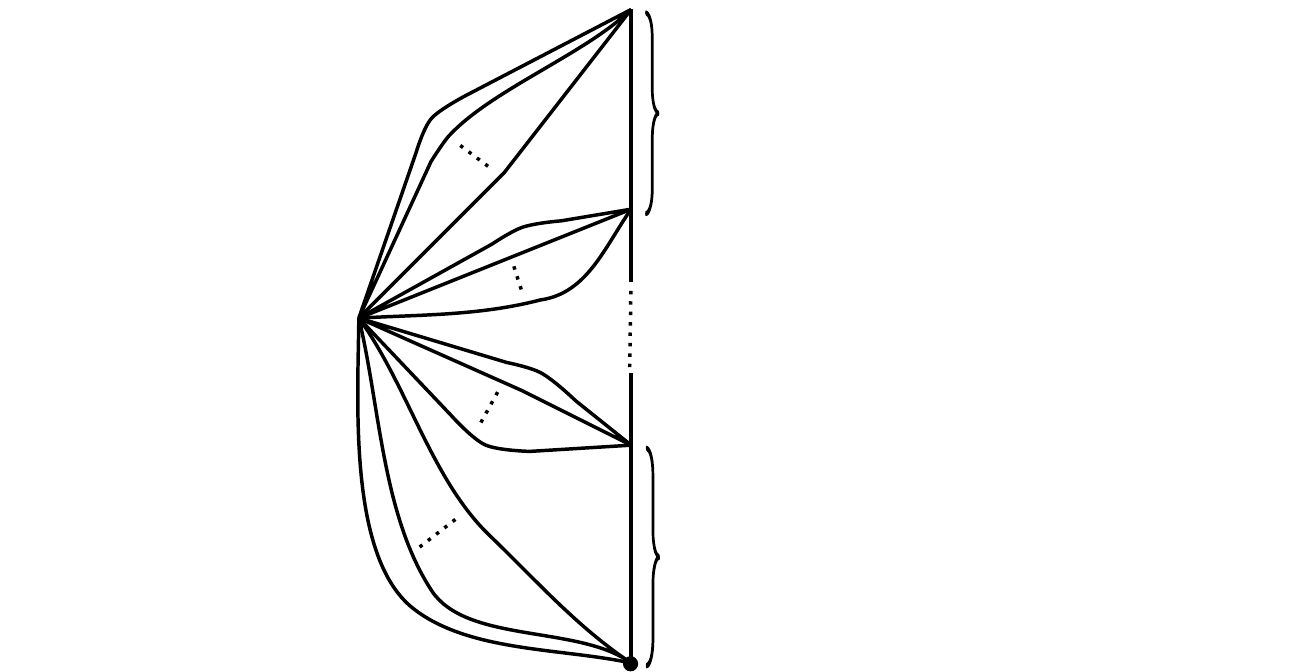
  \caption{Left: the state graph of $K[a_{n-1}, \dots, a_1]$ when $n$ is even (left) and odd (right). Right: the corresponding reduced state graphs.}
  \label{Fig:2BridgeStateGraph}
\end{figure}

Since we only consider orientable surfaces, we must have $b_{2j}$ even, for $j=1, \dots, k$. If $n$ is odd, then $b_{2k+2}$ must also be odd, else we cannot label the vertices $+$ and $-$ in an alternating manner. Replace the $b_{2k+2}$ edges in this case by one edge $e_{k+2}$, and by $b_{2k+2}-1$ additional edges with the same label, where $b_{2k+2}-1$ is even. Then the form of the reduced graph in the even and odd cases agree, so we can apply a single argument to both. Namely, there are regions  $R_0$ (unbounded), $R_1, \dots, R_k$ (or $R_{k+1}$), with $R_j$ bordered by edges $e_j$, $e_{j+1}$, and a collection of $b_{2j}$ consecutive edges, which we will denote by the strand $\epsilon_j$.

We may assume that the edge $e_1$ in $G'_{\sigma}$ has the same label as the other $b_2$ edges in the boundary of $R_1$, or else we may remove two crossings, as in \reflem{Consecutive}. Similarly, in the case $n=2k+3$ is odd, the edge $e_{k+1}$ has the same label as the other $b_{2k}$ edges in the region $R_k$.

We now consider generators of the fundamental group $\pi_1(G'_{\sigma})$.
First we set as basepoint the vertex $x$ on the left of the reduced graph $G'_{\sigma}$. The generators of $\pi_1(G'_{\sigma})$ are $\{\gamma_1,\dots, \gamma_k\}$ (or $\{\gamma_1, \dots, \gamma_{k+1}\}$), where $\gamma_j$ corresponds to the oriented boundary of the region $R_j$, oriented counterclockwise and based at $x$. Thus $\gamma_j$ runs from $x$ along $e_{j+1}$, then along $\epsilon_j$ from bottom to top, and then along $e_j$ back to $x$. As in \refsec{Stallings}, let $\{u_1, \dots, u_k\}$ (or $\{u_1, \dots, u_{k+1}\}$) denote the generators of $\pi_1(S^3-G'_\sigma)$.

\begin{lemma}\label{Lem:StallingsMap2Bridge}
The action of the Stallings map $\phi'\from\pi_1(G'_\sigma)\to\pi_2(S^3-G'_\sigma)$ is as follows. For any generator $\gamma_j$, described as above, $\phi'(\gamma_j)$ is the word $w_{j1}w_{j2}w_{j3}$ where
\[ w_{j1} = \begin{cases} u_j & \mbox{if $e_{j+1}$ is labeled $A$}\\
  u_{j+1} & \mbox{if $e_{j+1}$ is labeled $B$},
  \end{cases}
\]
\[ w_{j2} = \begin{cases} u_j^{b_{2j}/2} & \mbox{if edges of $\epsilon_j$ are labeled $A$} \\
  u_j^{-b_{2j}/2} & \mbox{if edges of $\epsilon_j$ are labeled $B$},
  \end{cases}
\]
\[ w_{j3} = \begin{cases} u_{j-1}^{-1} & \mbox{if $e_j$ is labeled $A$}\\
  u_j^{-1} & \mbox{if $e_j$ is labeled $B$}.
  \end{cases}
  \]
Here $u_0$ and $u_{k+1}$ (odd case) or $u_{k+2}$ (even case) are understood to be the identity.   
\end{lemma}

\begin{proof}
The proof is by applying \refprop{StallingsStar} to each $\gamma_j$. 
\end{proof}

To simplify, abelianize and consider the effect of the Stallings map on homology. 

\begin{lemma}\label{Lem:2BridgeMatrix}
The the matrix $M$ induced by the Stallings map $\phi'$ on homology is tridiagonal, i.e.\ it has the form 
\[
M = 
 \begin{pmatrix}
  p_1 & q_1 &   &  &\\
  r_1 & p_2 & q_2  &  &0 & \\
   & r_2 & p_3  &  &  \\
  &  &  & \ddots &  &  \\
  &0  &  &  &  p_{m-1}& q_{m-1} \\
  &  &  &  & r_{m-1} & p_m \\
 \end{pmatrix}
\]
and its determinant is given by 
\[ \det(M)=\prod_{i=1}^{m}p_i. \]
\end{lemma}

Notation in the lemma: $m=k$ when $n=2k+2$ is even, and $m=k+1$ when $n=2k+3$ is odd. 

\begin{proof}
The fact that the induced map on homology is tridiagonal follows immediately from \reflem{StallingsMap2Bridge}: each word $\phi(\gamma_j)$ involves at most the generators $u_{j-1}$, $u_j$, and $u_{j+1}$.

The terms of the matrix are integers that encode the powers of the generators $u_i$. Note that since the region $R_j$ shares a single edge with each of its neighboring regions, it must be the case that $|q_j|+|r_j|=1$. Thus the determinant of $M$ is $\prod_{i=1}^n p_i$.   
\end{proof}

\begin{lemma}\label{Lem:2BridgeFiber}
  Let $K[a_n, a_{n-1}, \dots, a_1]$ be a 2-bridge link with bounded checkerboard surface $S_\sigma$. Let $M$ denote the matrix of \reflem{2BridgeMatrix}. Then $S_\sigma$ is a fiber if and only if $|\det(M)|=1$.
\end{lemma}

\begin{proof}
By \refthm{Algebraic} and \reflem{MultipleEdges}, $S_\sigma$ is a fiber if and only if the Stallings map $\phi'\from \pi_1(G'_\sigma)\to\pi_1(S^3-G'_\sigma)$ is an isomorphism.

If $\phi'$ is an isomorphism, then it must be the case that $\det(M)=\pm 1$. 

We show the converse: if $\det(M)=\pm 1$ then $\phi'$ is an isomorphism.  In fact, if $\det(M)=\pm 1$, then $|p_i|=1$, for $i=1,\dots, n$. We proceed by induction on the regions $R_1,\dots, R_m$.

Recall that the $b_i$ are even, positive integers.

\textbf{Case I:} If $e_2$ is labeled $A$, there are two possibilities for $\phi'(\gamma_1)$:
\begin{enumerate}
\item[(a)] $\gamma_1 \mapsto u_1u_1^{b_2/2}=u_1^{b_2/2+1}$, if $e_1$ and the other edges of $R_1$ are labeled $A$, or

\item[(b)] $\gamma_1 \mapsto u_1u_1^{-b_2/2-1}=u_1^{-b_2/2}$, if $e_1$ and the other edges of $R_1$ are labeled $B$. 
\end{enumerate}

Since $|p_1|=1$, if (a) happens, we must have $b_2=0$, which we ruled out by our choices. If (b) happens, we must have $b_2=2$. In this case we are able to decompose a Hopf band. The decomposition of case (b) is illustrated in \reffig{2BridgeDecompose} (top).

\begin{figure}
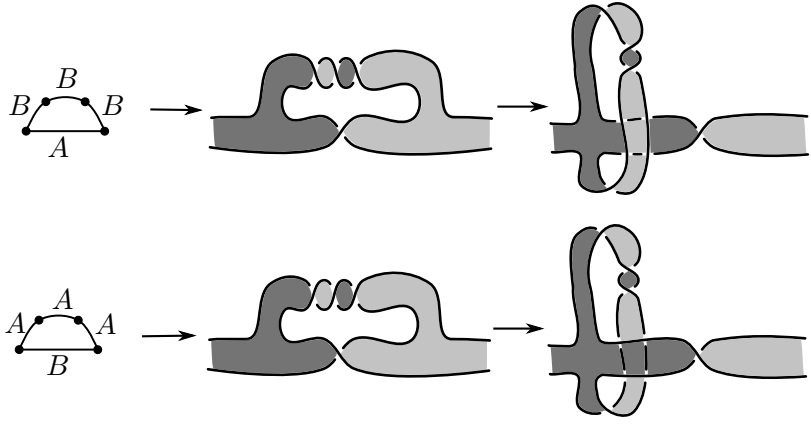
\caption{Top: Part (b) of Case I; Bottom: Part (a) of Case II.}
\label{Fig:2BridgeDecompose}
\end{figure}

\textbf{Case II:} If $e_2$ is labeled $B$, again there two possibilities for $\phi'(\gamma_1)$:
\begin{enumerate}
\item[(a)] $\gamma_1\mapsto u_2u_1^{b_2/2}$, if $e_1$ and the other edges of $R_1$ are labeled $A$,
\item[(b)] $\gamma_1\mapsto u_2u_1^{-b_2/2-1}$, if $e_1$ and the other edges of $R_1$ are labeled $B$.
\end{enumerate}

The argument is similar to the previous case: Since $|p_1|=1$, if (b) happens we must have $b_2=0$, which is impossible. If (a) happens we must have $b_2=2$, and we decompose a Hopf band. In (a) this decomposition is illustrated in \reffig{2BridgeDecompose} (bottom).

After these decompositions we are left with a graph consisting of the regions $R_2,\dots, R_n$. We proceed inductively, as before, in the edges of the boundary of each region. In each step we use the fact that $|p_i|=1$ to obtain the desired decompositions. We conclude that if $|\det(M)|=1$ then the surface $S_{\sigma}$ must be a fiber.
\end{proof}

\begin{theorem}\label{Thm:2Bridge}
Suppose $K$ is a 2-bridge link with diagram $K[a_{n-1}, \dots, a_1]$, where $|a_{n-1}|,|a_1|\geq 2$, and $|a_j|\geq 1$ for all $j$. Then the associated bounded checkerboard surface is a fiber if and only if the $a_i$ satisfy the following.
\begin{enumerate}
\item For odd indices $2j+1\neq n-2$, $a_{2j+1}$ can be any nonzero integer. 
\item For even indices $2j$ with $2j<n-1$:
  \begin{itemize}
  \item If $a_{2j-1}>0$ and $a_{2j+1}>0$, then $a_{2j}=-4$.
  \item If $a_{2j-1}>0$ and $a_{2j+1}<0$, then $a_{2j} = \pm 2$.
  \item If $a_{2j-1}<0$ and $a_{2j+1}>0$, then $a_{2j} = \pm 2$.
  \item If $a_{2j-1}<0$ and $a_{2j+1}<0$, then $a_{2j} = 4$.
  \end{itemize}
\item If $(n-1)=2k+2$ is even, then either $a_{2k+1}>0$ and $a_{2k+2}=-3$, or $a_{2k+1}<0$ and $a_{2k+2}=3$. 
\end{enumerate}
\end{theorem}

\begin{proof}
By \reflem{2BridgeFiber}, it suffices to find all choices of the $a_k$ for which the tridiagonal matrix of \reflem{2BridgeMatrix} has determinant $\pm 1$. To do so, we consider the images of the generators $\gamma_j$ as in \reflem{StallingsMap2Bridge}. For each generator, aside from the first and the last, there are eight choices of labels $A$ and $B$ for the edges meeting $\gamma_j$, each determining an integer $p_j$. When we set the $p_j$ to be $\pm 1$ we obtain the result.

For example, when $e_j$, $\epsilon_j$ and $e_{j+1}$ are all of type $A$, meaning $a_{2j-1}$, $a_{2j}$, and $a_{2j+1}$ are positive, we obtain $p_j = b_{2j}/2+1$. If this is $+1$, then $b_{2j} = |a_{2j}|$ must be zero, which is ruled out by our assumptions on $a_{2j}$. If it is $-1$, then $b_{2j}=|a_{2j}| = -4$, which is impossible. So $a_{2j-1}$, $a_{2j}$, and $a_{2j+1}$ cannot all be positive. If $a_{2j-1}$ and $a_{2j+1}$ are positive (type $A$) and $a_{2j}$ is negative (type $B$), then $p_j = -b_{2j}/2+1$. This cannot be $+1$ by our assumption that $b_{2j}=|a_{2j}|>0$, hence it equals $-1$ and $b_{2j} = 4$. The other cases are dealt with similarly.

Finally, when $(n-1)=2k+2$ is even, the value of $p_{k+1}$ depends on the signs of $a_{2k+1}$ and $a_{2k+2}$. When they are both positive or both negative, we find that $p_{k+1} = \pm ((b_{2k+2}-1)/2+1)$, which is $\pm 1$ only if $b_{2k+2}=|a_{2k+2}| = 1$. This is ruled out by our assumption that $|a_{n-1|}=|a_{2k+2}|\geq 2$. When $a_{2k+1}>0$ and $a_{2k+2}<0$, $p_{k+1} = -(b_{2k+2}-1)/2$, which is $-1$ when $b_{2k+2} =|a_{2k+2}|=3$. Similarly for $a_{2k+1}<0$ and $a_{2k+2}>0$. 
\end{proof}

\bibliographystyle{amsplain}
\bibliography{references}

\end{document}